\documentclass[letterpaper]{article}
\usepackage{uai2018}
\usepackage[margin=1in]{geometry}

\usepackage{times}

\title{Identification of Strong Edges in AMP Chain Graphs}

\author{ {\bf Jose M. Pe\~{n}a} \\
Department of Computer and Information Science \\
Link{\"o}ping University\\
58183 Link{\"o}ping, Sweden\\
}

\usepackage{graphicx,amssymb,float,MnSymbol,tikz,amsthm}
\usepackage[round]{natbib}
\usetikzlibrary{arrows}
\usetikzlibrary{decorations.markings}

\newtheorem{theorem}{Theorem}
\newtheorem{lemma}{Lemma}
\newtheorem{corollary}{Corollary}

\def\ci{\!\perp\!}

\def\ra{\rightarrow}
\def\la{\leftarrow}

\def\bb{\leftfootline\!\!\!\!\!\rightfootline}
\def\bo{\leftfootline\!\!\!\!\!\multimap}

\def\ob{\mathrel{\reflectbox{\ensuremath{\bo}}}}

\def\bn{\leftfootline}
\def\nb{\rightfootline}
\def\oo{\mathrel{\reflectbox{\ensuremath{\multimap}}}\!\!\!\!\!\multimap}

\def\no{\multimap}

\newcommand{\comments}[1]{}

\tikzset{tt/.style={decoration={
  markings,
  mark=at position .485 with {\arrow{>}},
  mark=at position .515 with {\arrow{<}}},postaction={decorate}}}

\begin{document}

\maketitle

\begin{abstract}
The essential graph is a distinguished member of a Markov equivalence class of AMP chain graphs. However, the directed edges in the essential graph are not necessarily strong or invariant, i.e. they may not be shared by every member of the equivalence class. Likewise for the undirected edges. In this paper, we develop a procedure for identifying which edges in an essential graph are strong. We also show how this makes it possible to bound some causal effects when the true chain graph is unknown.
\end{abstract}

\maketitle

\section{INTRODUCTION}

In most practical applications, the data available consists of observations. Therefore, it can rarely single out the true causal model. At best, it identifies the Markov equivalence class that contains the true causal model. In this paper, we represent causal models with the help of AMP chain graphs \citep{Anderssonetal.2001}. As argued by \citet{Penna2016}, these graphs are suitable for representing causal linear models with additive Gaussian noise. Intuitively, the directed subgraph of a chain graph represents the causal relations in the domain, and the undirected subgraph represents the dependence structure of the noise terms. Additive noise is a rather common assumption in causal discovery \citep{Petersetal.2017}, mainly because it produces tractable models which are useful for gaining insight into the system under study. Note also that linear structural equation models, which have extensively been studied for causal effect identification \citep{Pearl2009}, are additive noise models.

In order to represent the equivalence class of chain graphs identified from the observations at hand, we typically use a distinguished member of it. In the literature, there are two distinguished members: The essential graph \citep{AnderssonandPerlman2006}, and the largest deflagged graph \citep{RoveratoandStudeny2006}. In general, they do not coincide: The essential graph is a deflagged graph \citep[Lemma 3.2]{AnderssonandPerlman2006} but not necessarily the largest in the equivalence class \citep[p. 57]{Anderssonetal.2001}. Unfortunately, the directed edges in either of the two representatives are not necessarily strong,\footnote{The term invariant or essential is also used in the literature.} i.e. they may not be shared by every member of the equivalence class. Likewise for the undirected edges. In this paper, we use essential graphs to represent equivalence classes of chain graphs. And we develop a procedure for identifying which edges in an essential graph are strong. Note that while we assume that the true chain graph is unknown, its corresponding essential graph can be obtained from observational data as follows. First, learn a chain graph as shown by \citet{Penna2014,Penna2016} and \citet{PennaandGomezOlmedo2016} and, then, transform it into an essential graph as shown by \citet[Section 3]{SonntagandPenna2015}.

Identifying the strong edges in an essential graph is important because it makes it possible to identify causal paths from data even though the data may not be able to single out the true chain graph: Simply output every directed path in the essential graph that consists of only strong edges. Of course, the true chain graph may have additional causal paths. Identifying the strong edges in an essential graph is also important because it allows to efficiently bound some causal effects of the form $p(y | do(x))$ where $X$ and $Y$ are singletons. The simplest way to bound such a causal effect consists in enumerating all the chain graphs that are equivalent to the essential graph and, then, computing the causal effect for each of them from the observational data by adjusting for the appropriate variables. Although we know how to enumerate the equivalent chain graphs \citep[Theorem 3]{SonntagandPenna2015}, this method may be inefficient for all but small domains. Instead, we show in this paper how the knowledge of the strong edges in an essential graph allows to enumerate the adjusting sets without enumerating the equivalent chain graphs explicitly.

The rest of the paper is organized as follows. Section \ref{sec:preliminaries} introduces some preliminaries. Section \ref{sec:strong} presents our algorithm to identify strong edges in an essential graph. Section \ref{sec:bounds} presents our procedure to bound causal effects when the true chain graph is unknown but its corresponding essential graph is known. Section \ref{sec:discussion} closes the paper with some discussion and lines of future research.

\section{PRELIMINARIES}\label{sec:preliminaries}

All the graphs and probability distributions in this paper are defined over a finite set $V$ unless otherwise stated. All the graphs contain at most one edge between a pair of nodes. The elements of $V$ are not distinguished from singletons.

The parents of a set of nodes $X$ of a graph $G$ is the set $Pa_G(X) = \{A | A \ra B$ is in $G$ with $B \in X \}$. The children of $X$ is the set $Ch_G(X) = \{A | B \ra A$ is in $G$ with $B \in X \}$. The neighbors of $X$ is the set $Ne_G(X) = \{A | A - B$ is in $G$ with $B \in X \}$. The adjacents of $X$ is the set $Ad_G(X) = \{A | A \ra B$, $B \ra A$ or $A - B$ is in $G$ with $B \in X \}$. The descendants of $X$ is the set $De_G(X) = \{A | B \ra \cdots \ra A$ is in $G$ with $B \in X \}$. A route from a node $V_{1}$ to a node $V_{n}$ in $G$ is a sequence of (not necessarily distinct) nodes $V_{1}, \ldots, V_{n}$ such that $V_i \in Ad_G(V_{i+1})$ for all $1 \leq i < n$. A route is called a cycle if $V_n=V_1$. A cycle has a chord if two non-consecutive nodes of the cycle are adjacent in $G$. A cycle is called semidirected if it is of the form $V_1 \ra V_2 \no \cdots \no V_n$ where $\no$ is a short for $\ra$ or $-$. A chain graph (CG) is a graph with (possibly) directed and undirected edges, and without semidirected cycles. A set of nodes of a CG $G$ is connected if there exists a route in $G$ between every pair of nodes in the set and such that all the edges in the route are undirected. A chain component of $G$ is a maximal connected set. Note that the chain components of $G$ can be sorted topologically, i.e. for every edge $A \ra B$ in $G$, the component containing $A$ precedes the component containing $B$. A set of nodes of $G$ is complete if there is an undirected edge between every pair of nodes in the set. Moreover, a node is called simplicial if its neighbors are a complete set.

We now recall the interpretation of CGs due to \citet{Anderssonetal.2001}, also known as AMP CGs.\footnote{\citet{Anderssonetal.2001} interpret CGs via the so-called augmentation criterion. \citet[Theorem 4.1]{Levitzetal.2001} introduce the so-called p-separation criterion and prove its equivalence to the augmentation criterion. \citet[Theorem 2]{Penna2016} introduce the route-based criterion that we use in this paper and prove its equivalence to the p-separation criterion.} A node $B$ in a route $\rho$ in a CG $G$ is called a triplex node in $\rho$ if $A \ra B \la C$, $A \ra B - C$, or $A - B \la C$ is a subroute of $\rho$. Moreover, $\rho$ is said to be $Z$-open with $Z \subseteq V$ when (i) every triplex node in $\rho$ is in $Z$, and (ii) every non-triplex node in $\rho$ is outside $Z$. Let $X$, $Y$ and $Z$ denote three disjoint subsets of $V$. When there is no $Z$-open route in $G$ between a node in $X$ and a node in $Y$, we say that $X$ is separated from $Y$ given $Z$ in $G$ and denote it as $X \ci_G Y | Z$. The statistical independences represented by $G$ are the separations $X \ci_G Y | Z$. A probability distribution $p$ is Markovian with respect to $G$ if the independences represented by $G$ are a subset of those in $p$. If the two sets of independences coincide, then $p$ is faithful to $G$. Two CGs are Markov equivalent if the sets of distributions that are Markovian with respect to each CG are the same. If a CG has an induced subgraph of the form $A \ra B \la C$, $A \ra B - C$ or $A - B \la C$, then we say that the CG has a triplex $(A,B,C)$. Two CGs are Markov equivalent if and only if they have the same adjacencies and triplexes \citep[Theorem 5]{Anderssonetal.2001}.

\begin{lemma}\label{lem:eq}
Two CGs $G$ and $H$ are Markov equivalent if and only if they represent the same independences.
\end{lemma}

\begin{proof}
The if part is trivial. To see the only if part, note that \citet[Theorem 6.1]{Levitzetal.2001} prove that there are Gaussian distributions $p$ and $q$ that are faithful to $G$ and $H$, respectively. Moreover, $p$ is Markovian with respect to $H$, because $G$ and $H$ are Markov equivalent. Likewise for $q$ and $G$. Therefore, $G$ and $H$ must represent the same independences.
\end{proof}

\subsection{ESSENTIAL GRAPHS}

The essential graph (EG) $G^*$ is a distinguished member of a class of equivalent CGs. Specifically, an edge $A \ra B$ is in $G^*$ if and only if $A \ra B$ is in some member of the class and $A \la B$ is in no member of the class. An algorithm (without proof of correctness) for constructing the EG from any other member of the equivalence class has been developed by \citet[Section 7]{AnderssonandPerlman2004}. An alternative algorithm with proof of correctness has been developed by \citet[Section 3]{SonntagandPenna2015}. The latter algorithm can be seen in Tables \ref{tab:algorithmR} and \ref{tab:rulesR}. A perpendicular line  at the end of an edge such as in $\bn$ or $\bb$ represents a block, and it means that the edge cannot be oriented in that direction. Note that the ends of some of the edges in the rules in Table \ref{tab:rulesR} are labeled with a circle such as in $\bo$ or $\oo$. The circle represents an unspecified end, i.e. a block or nothing. The modifications in the consequents of the rules consist in adding some blocks. Note that only the blocks that appear in the consequents are added, i.e. the circled ends do not get modified. In line 2 of Table \ref{tab:algorithmR}, any such set $S$ will do. For instance, if $B \notin De_G(A)$, then let $S=Ne_G(A) \cup Pa_G(A \cup Ne_G(A))$, otherwise let $S=Ne_G(B) \cup Pa_G(B \cup Ne_G(B))$. In line 5, that the cycle has no blocks means that the ends of the edges in the cycle have no blocks. Note that the rule R1 is not used in line 6, because it will never fire after its repeated application in line 4. Finally, note that $G^*$ may have edges without blocks after line 6.

\begin{table}[t]
\caption{Algorithm for constructing the EG.}\label{tab:algorithmR}
\begin{center}
\scalebox{1}{
\begin{tabular}{|rl|}
\hline
& In: A CG $G$.\\
& Out: The EG $G^*$ in the equivalence class of $G$.\\
&\\
1 & For each ordered pair of non-adjacent nodes $A$\\
& and $B$ in $G$\\ 
2 & \hspace{0.2cm} Set $S_{AB}=S_{BA}=S$ such that $A \ci_G B | S$\\
3 & Let $G^*$ denote the undirected graph that has the\\
& same adjacencies as $G$\\
4 & Apply the rules R1-R4 to $G^*$ while possible\\
5 & Replace every edge $A - B$ in every cycle in $G^*$\\
& that is of length greater than three, chordless,\\
& and without blocks with $A \bb B$\\
6 & Apply the rules R2-R4 to $G^*$ while possible\\
7 & Replace every edge $A \bn B$ and $A \bb B$ in $G^*$\\
& with $A \ra B$ and $A - B$, respectively\\
\hline
\end{tabular}}
\end{center}
\end{table}

\begin{table}[t]
\caption{Rules in the algorithm in Table \ref{tab:algorithmR}. The antecedents represent induced subgraphs.}\label{tab:rulesR}
\begin{center}
\scalebox{0.75}{
\begin{tabular}{|cccc|}
\hline
R1:&
\begin{tabular}{c}
\begin{tikzpicture}[inner sep=1mm]
\node at (0,0) (A) {$A$};
\node at (1.5,0) (B) {$B$};
\node at (3,0) (C) {$C$};
\path[o-o] (A) edge (B);
\path[o-o] (B) edge (C);
\end{tikzpicture} 
\end{tabular}
& $\Rightarrow$ &
\begin{tabular}{c}
\begin{tikzpicture}[inner sep=1mm]
\node at (0,0) (A) {$A$};
\node at (1.5,0) (B) {$B$};
\node at (3,0) (C) {$C$};
\path[|-o] (A) edge (B);
\path[o-|] (B) edge (C);
\end{tikzpicture}
\end{tabular}\\
& and $B \notin S_{AC}$&&\\
&&&\\
\hline
&&&\\
R2:&
\begin{tabular}{c}
\begin{tikzpicture}[inner sep=1mm]
\node at (0,0) (A) {$A$};
\node at (1.5,0) (B) {$B$};
\node at (3,0) (C) {$C$};
\path[|-o] (A) edge (B);
\path[o-o] (B) edge (C);
\end{tikzpicture}
\end{tabular}
& $\Rightarrow$ &
\begin{tabular}{c}
\begin{tikzpicture}[inner sep=1mm]
\node at (0,0) (A) {$A$};
\node at (1.5,0) (B) {$B$};
\node at (3,0) (C) {$C$};
\path[|-o] (A) edge (B);
\path[|-o] (B) edge (C);
\end{tikzpicture}
\end{tabular}\\
& and $B \in S_{AC}$&&\\
&&&\\
\hline
&&&\\
R3:&
\begin{tabular}{c}
\begin{tikzpicture}[inner sep=1mm]
\node at (0,0) (A) {$A$};
\node at (1.5,0) (B) {$\ldots$};
\node at (3,0) (C) {$B$};
\path[|-o] (A) edge (B);
\path[|-o] (B) edge (C);
\path[o-o] (A) edge [bend left] (C);
\end{tikzpicture}
\end{tabular}
& $\Rightarrow$ &
\begin{tabular}{c}
\begin{tikzpicture}[inner sep=1mm]
\node at (0,0) (A) {$A$};
\node at (1.5,0) (B) {$\ldots$};
\node at (3,0) (C) {$B$};
\path[|-o] (A) edge (B);
\path[|-o] (B) edge (C);
\path[|-o] (A) edge [bend left] (C);
\end{tikzpicture}
\end{tabular}\\
&&&\\
\hline
&&&\\
R4:&
\begin{tabular}{c}
\begin{tikzpicture}[inner sep=1mm]
\node at (0,0) (A) {$A$};
\node at (2,0) (B) {$B$};
\node at (1,1) (C) {$C$};
\node at (1,-1) (D) {$D$};
\path[o-o] (A) edge (B);
\path[o-o] (A) edge (C);
\path[o-o] (A) edge (D);
\path[|-o] (C) edge (B);
\path[|-o] (D) edge (B);
\end{tikzpicture}
\end{tabular}
& $\Rightarrow$ &
\begin{tabular}{c}
\begin{tikzpicture}[inner sep=1mm]
\node at (0,0) (A) {$A$};
\node at (2,0) (B) {$B$};
\node at (1,1) (C) {$C$};
\node at (1,-1) (D) {$D$};
\path[|-o] (A) edge (B);
\path[o-o] (A) edge (C);
\path[o-o] (A) edge (D);
\path[|-o] (C) edge (B);
\path[|-o] (D) edge (B);
\end{tikzpicture}
\end{tabular}\\
& and $A \in S_{CD}$&&\\
\hline
\end{tabular}}
\end{center}
\end{table}

\section{STRONG EDGES}\label{sec:strong}

We say that a directed edge in a CG is strong if it appears in every equivalent CG. Likewise for undirected edges. Therefore, strong edges are features of a class of equivalent CGs. Clearly, strong directed edges correspond to directed edges in the EG of the equivalence class. However, the opposite is not true. Likewise for strong undirected edges. For an example, consider the EG $A \ra B \la C - D$. The naive way to detect which edges in an EG are strong consists in generating all the CGs in the equivalence class and, then, recording the shared edges. Since there may be many CGs in the equivalence class, enumerating them in an efficient manner is paramount, but challenging. In truth, it suffices to enumerate what we call the minimally oriented CGs in order to identify the strong directed edges and, then, find one maximally oriented CG to identify the strong undirected edges. We prove these claims in Section \ref{sec:minimallyoriented}. Although there are typically considerably fewer minimally oriented CGs, enumerating them in an efficient manner seems challenging too. That is why we present in Section \ref{sec:algorithmS} an algorithm that does not rely on enumerating CGs or minimally oriented CGs.

\subsection{MINIMALLY AND MAXIMALLY ORIENTED CGs}\label{sec:minimallyoriented}

Given a CG $G$, merging two of its chain components $U$ and $L$ implies replacing the edge $A \ra B$ with $A - B$ for all $A \in U$ and $B \in L$. We say that a merging is feasible when
\begin{enumerate}
\item $L \subseteq Ch_G(X)$ for all $X \in Pa_G(L) \cap U$,
\item $Pa_G(L) \cap U$ is a complete set,
\item $Pa_G(Pa_G(L) \cap U) \subseteq Pa_G(Y)$ for all $Y \in L$, and
\item $De_G(U) \cap Pa_G(L) = \emptyset$.
\end{enumerate}

A feasible merging of two chain components of a CG results in an equivalent CG \citep[Lemma 2]{SonntagandPenna2015}. If a CG does not admit any feasible merging, then we call it minimally oriented. Note that several equivalent minimally oriented CGs may exist, e.g. $A \ra B - C$ and $A - B \la C$. Note also that an EG is not necessarily a minimally oriented CG, e.g. $A \ra B \la C$. If the directed edges of a CG are a subset of the directed edges of a second CG (with the same orientation), then we say that the former is larger than the latter. 

\begin{lemma}\label{lem:largest}
The minimally oriented CGs in an equivalence class are the maximally large CGs in the class, and vice versa.
\end{lemma}

\begin{proof}
Clearly, a maximally large CG must be minimally oriented because, otherwise, it admits a feasible merging which results in a larger CG, which is a contradiction. On the other hand, let $G$ be a minimally oriented CG, and assume to the contrary that there is a CG $H$ that is equivalent but larger than $G$. Specifically, let $G$ have an edge $A \ra B$ whereas $H$ has an edge $A - B$. Consider a topological ordering of the chain components of $G$. We say that an edge $X \ra Y$ precedes an edge $Z \ra W$ in $G$ if the chain component of $X$ precedes the chain component of $Z$ in the ordering, or if both chain components coincide and the chain component of $Y$ precedes the chain component of $W$ in the ordering. Assume without loss of generality that no other edge that is directed in $G$ but undirected in $H$ precedes the edge $A \ra B$ in $G$. Let $U$ and $L$ denote the chain components of $A$ and $B$, respectively. Clearly, all the directed edges from $U$ to $L$ in $G$ must be undirected in $H$ because, otherwise, $H$ has a semidirected cycle. However, this implies a contradiction. To see it, recall that $G$ is a minimally oriented CG and, thus, merging $U$ and $L$ in $G$ is not feasible. If condition 1 fails, then $G$ has an induced subgraph $X \ra Y - Z$ where $X \in U$ and $Y, Z \in L$, whereas $H$ has an induced subgraph $X - Y - Z$. However, this implies that $G$ and $H$ are not equivalent, since $G$ has a triplex $(X,Y,Z)$ that $H$ has not.

If condition 2 fails but condition 1 holds, then $G$ has an induced subgraph $X \ra Y \la Z$ where $X, Z \in U$ and $Y \in L$, whereas $H$ has an induced subgraph $X - Y - Z$. However, this implies that $G$ and $H$ are not equivalent, since $G$ has a triplex $(X,Y,Z)$ that $H$ has not.

If condition 3 fails but condition 1 holds, then $G$ has an induced subgraph $Z \ra X \ra Y$ where $X \in U$, $Y \in L$ and $Z \in V \setminus (U \cup L)$, whereas $H$ has an induced subgraph $Z \ra X - Y$. However, this implies that $G$ and $H$ are not equivalent, since $H$ has a triplex $(Z,X,Y)$ that $G$ has not. Note that $Z \ra X$ is in $H$ because $Z \ra X$ precedes $X \ra Y$ and thus $A \ra B$ in $G$.

Finally, if condition 4 fails but condition 1 holds, then $G$ has a subgraph of the form $X \ra Y \la \cdots \la Z \la X' - \cdots - X$ where $X,X' \in U$, $Y \in L$ and $Z \in V \setminus (U \cup L)$, whereas $H$ has a subgraph of the form $X - Y - \cdots - Z - X$. To see it, note that any other option results in a semidirected cycle because, recall, $H$ is larger than $G$. However, this is a contradiction because $X' \ra Z$ precedes $X \ra Y$ and thus $A \ra B$ in $G$.
\end{proof}

The following result follows from the previous lemma.

\begin{theorem}
A directed edge is strong if and only if it is in every minimally oriented CG in the equivalence class.
\end{theorem}

Finally, one may think that an undirected edge that is in every minimally oriented CG in the equivalence class is strong. But this is not true. For an example, consider the equivalence class represented by the EG $A - B$. Instead, an undirected edge is strong if and only if it is in any maximally oriented CG in the equivalence class \citep[Theorems 4 and 5]{SonntagandPenna2015}. Formally, a maximally oriented CG is a CG that does not admit any feasible split, which is the inverse operation of the feasible merge operation described before. Alternatively, we can say that if the minimally oriented CGs are the maximally large CGs in an equivalence class, then the maximally oriented CGs are the minimally large \citep[Lemma 13]{SonntagandPenna2015}. Note that several equivalent maximally oriented CGs may exist (e.g., $A \ra B$ and $A \la B$) but all of them have the same undirected edges \citep[Theorems 4 and 5]{SonntagandPenna2015}. Note also that an EG is not necessarily a maximally oriented CG, e.g. $A - B$.

\begin{table}[t]
\caption{Algorithm to label strong edges in an EG. It replaces line 7 of the algorithm in Table \ref{tab:algorithmR}.}\label{tab:algorithmS}
\begin{center}
\scalebox{1}{
\begin{tabular}{|rl|}
\hline
7 & Label every edge $X \bb Y$ as strong in $G^*$\\
8 & For each edge $X \bn Y$ in $G^*$\\
9 & \hspace{0.2cm} Set $H = G^*$\\
10 & \hspace{0.2cm} Replace $X \bn Y$ in $H$ with $X \bb Y$\\
11 & \hspace{0.2cm} Apply the rules R2-3 to $H$ while possible\\
12 & \hspace{0.2cm} If $G^*$ has an induced subgraph $A \bn B \ob C$\\
& \hspace{0.2cm} whereas $H$ has $A \bb B \bb C$ then\\
13 & \hspace{0.5cm} Label $X \bn Y$ as strong in $G^*$\\
14 & Replace every edge $X \bn Y$ and $X \bb Y$ in $G^*$\\
& with $X \ra Y$ and $X - Y$, respectively\\
\hline
\end{tabular}}
\end{center}
\end{table}

\subsection{ENUMERATION-FREE ALGORITHM}\label{sec:algorithmS}

Although the minimally and maximally oriented CGs in an equivalence class can be obtained by repeatedly performing feasible splits and merges \citep[Theorem 3]{SonntagandPenna2015}, the approach outlined above for identifying strong edges via enumeration may be inefficient for all but small domains. Hence, Table \ref{tab:algorithmS} presents an alternative algorithm that does not rely on enumerating the CGs or the minimally oriented CGs in the equivalence class. The new algorithm replaces line 7 in Table \ref{tab:algorithmR}. In other words, the new algorithm postpones orienting edges until line 14, and in lines 7-13 it identifies which of the future directed and undirected edges are strong. Line 7 identifies the strong undirected edges, whereas lines 8-13 identify the strong directed edges. To do the latter, the algorithm tries to build a CG $H$ that is equivalent to $G^*$ and contains an edge $X - Y$. If this fails, then $X \ra Y$ is strong. Specifically, line 10 forces the edge between $X$ and $Y$ to be undirected in $H$ by blocking the end at $Y$. Line 11 computes other blocks that follow from the new block at $Y$. After line 11, $H$ can be oriented as indicated in line 14 without creating a semidirected cycle or a triplex that is not in $G^*$. Finally, line 12 checks if every triplex in $G^*$ is in $H$. If not, $X - Y$ is incompatible with some triplex in $G^*$, which implies that $X \ra Y$ is strong in $G^*$. We prove the correctness of the algorithm below.

\begin{lemma}\label{lem:noundirectedcycleS}
After line 11, $H$ does not have any induced subgraph of the form
\begin{tabular}{c}
\scalebox{0.75}{
\begin{tikzpicture}[inner sep=1mm]
\node at (0,0) (A) {$A$};
\node at (1,0) (B) {$B$};
\node at (2,0) (C) {$C$};
\path[|-o] (A) edge (B);
\path[-] (B) edge (C);
\path[-] (A) edge [bend left] (C);
\end{tikzpicture}.}
\end{tabular}
\end{lemma}

\begin{proof}
The proof is an adaptation of the proof of Lemma 5 by \citet{Penna2014}. Assume to the contrary that the lemma does not hold. We interpret the execution of lines 10-11 as a sequence of block additions and, for the rest of the proof, one particular sequence of these block additions is fixed. Fixing this sequence is a crucial point upon which some important later steps of the proof are based. Since there may be several induced subgraphs of $H$ of the form under study after lines 10-11, let us consider any of the induced subgraphs
\begin{tabular}{c}
\scalebox{0.75}{
\begin{tikzpicture}[inner sep=1mm]
\node at (0,0) (A) {$A$};
\node at (1,0) (B) {$B$};
\node at (2,0) (C) {$C$};
\path[|-o] (A) edge (B);
\path[-] (B) edge (C);
\path[-] (A) edge [bend left] (C);
\end{tikzpicture}}
\end{tabular}
that appear first during the execution of lines 10-11 and fix it for the rest of the proof. Note that $H$ has no such induced subgraph after line 9 \citep[Lemma 9]{SonntagandPenna2015}. Now, consider the following cases.

\begin{description}

\item[Case 1] Assume that $A \bo B$ is in $H$ due line 10. However, this implies that $H$ had an induced subgraph\begin{tabular}{c}
\scalebox{0.75}{
\begin{tikzpicture}[inner sep=1mm]
\node at (0,0) (A) {$A$};
\node at (1,0) (B) {$B$};
\node at (2,0) (C) {$C$};
\path[-|] (A) edge (B);
\path[-] (B) edge (C);
\path[-] (A) edge [bend left] (C);
\end{tikzpicture}}
\end{tabular}
before line 10, which is a contradiction \citep[Lemma 9]{SonntagandPenna2015}.

\item[Case 2] Assume that $A \bo B$ is in $H$ due to R2 in line 11. Then, after line 11, $H$ has an induced subgraph of one of the following forms: 

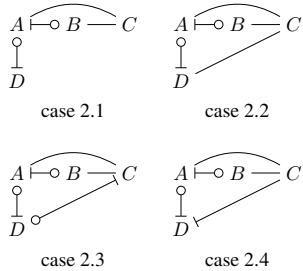
\begin{table}[H]
\centering
\scalebox{0.75}{
\begin{tabular}{cc}
\begin{tikzpicture}[inner sep=1mm]
\node at (0,0) (A) {$A$};
\node at (1,0) (B) {$B$};
\node at (2,0) (C) {$C$};
\node at (0,-1) (D) {$D$};
\path[|-o] (A) edge (B);
\path[-] (B) edge (C);
\path[-] (A) edge [bend left] (C);
\path[|-o] (D) edge (A);
\end{tikzpicture}
&
\begin{tikzpicture}[inner sep=1mm]
\node at (0,0) (A) {$A$};
\node at (1,0) (B) {$B$};
\node at (2,0) (C) {$C$};
\node at (0,-1) (D) {$D$};
\path[|-o] (A) edge (B);
\path[-] (B) edge (C);
\path[-] (A) edge [bend left] (C);
\path[|-o] (D) edge (A);
\path[-] (D) edge (C);
\end{tikzpicture}\\
case 2.1&case 2.2\\
\\
\begin{tikzpicture}[inner sep=1mm]
\node at (0,0) (A) {$A$};
\node at (1,0) (B) {$B$};
\node at (2,0) (C) {$C$};
\node at (0,-1) (D) {$D$};
\path[|-o] (A) edge (B);
\path[-] (B) edge (C);
\path[-] (A) edge [bend left] (C);
\path[|-o] (D) edge (A);
\path[o-|] (D) edge (C);
\end{tikzpicture}
&
\begin{tikzpicture}[inner sep=1mm]
\node at (0,0) (A) {$A$};
\node at (1,0) (B) {$B$};
\node at (2,0) (C) {$C$};
\node at (0,-1) (D) {$D$};
\path[|-o] (A) edge (B);
\path[-] (B) edge (C);
\path[-] (A) edge [bend left] (C);
\path[|-o] (D) edge (A);
\path[|-] (D) edge (C);
\end{tikzpicture}\\
case 2.3&case 2.4
\end{tabular}}
\end{table}

\begin{description}

\item[Case 2.1] If $A \notin S_{CD}$ then $A \nb C$ is in $H$ by R1 in line 4 of Table \ref{tab:algorithmR}, else $A \bn C$ is in $H$ by R2. Either case is a contradiction.

\item[Case 2.2] Note that
\begin{tabular}{c}
\scalebox{0.75}{
\begin{tikzpicture}[inner sep=1mm]
\node at (0,0) (A) {$D$};
\node at (1,0) (B) {$A$};
\node at (2,0) (C) {$C$};
\path[|-o] (A) edge (B);
\path[-] (B) edge (C);
\path[-] (A) edge [bend left] (C);
\end{tikzpicture}}
\end{tabular}
cannot be an induced subgraph of $H$ after line 11 because, otherwise, it would contradict the assumption that
\begin{tabular}{c}
\scalebox{0.75}{
\begin{tikzpicture}[inner sep=1mm]
\node at (0,0) (A) {$A$};
\node at (1,0) (B) {$B$};
\node at (2,0) (C) {$C$};
\path[|-o] (A) edge (B);
\path[-] (B) edge (C);
\path[-] (A) edge [bend left] (C);
\end{tikzpicture}}
\end{tabular}
is one of the first induced subgraph of that form that appeared during the execution of lines 10-11. So, this case is impossible.

\item[Case 2.3] Note that $A \nb C$ is in $H$ by R3, which is a contradiction.

\item[Case 2.4] If $C \notin S_{BD}$ then $B \bn C$ is in $H$ by R1 in line 4 of Table \ref{tab:algorithmR}, else $B \nb C$ is in $H$ by R2. Either case is a contradiction.

\end{description}

\item[Case 3] Assume that $A \bo B$ is in $H$ due to R3 in line 11. Then, after line 11, $H$ had a subgraph of one of the following forms, where possible additional edges between $C$ and internal nodes of the route $A \bo \cdots \bo D$ are not shown:

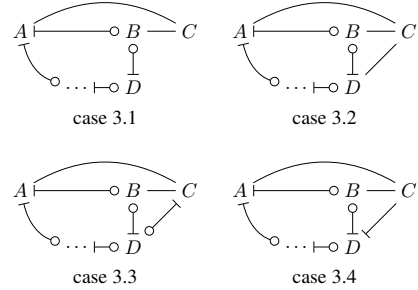
\begin{table}[H]
\centering
\scalebox{0.75}{
\begin{tabular}{cc}
\begin{tikzpicture}[inner sep=1mm]
\node at (-1,0) (A) {$A$};
\node at (1,0) (B) {$B$};
\node at (2,0) (C) {$C$};
\node at (1,-1) (D) {$D$};
\node at (0,-1) (E) {$\ldots$};
\path[|-o] (A) edge (B);
\path[-] (B) edge (C);
\path[-] (A) edge [bend left] (C);
\path[|-o] (A) edge [bend right] (E);
\path[|-o] (E) edge (D);
\path[|-o] (D) edge (B);
\end{tikzpicture}
&
\begin{tikzpicture}[inner sep=1mm]
\node at (-1,0) (A) {$A$};
\node at (1,0) (B) {$B$};
\node at (2,0) (C) {$C$};
\node at (1,-1) (D) {$D$};
\node at (0,-1) (E) {$\ldots$};
\path[|-o] (A) edge (B);
\path[-] (B) edge (C);
\path[-] (A) edge [bend left] (C);
\path[|-o] (A) edge [bend right] (E);
\path[|-o] (E) edge (D);
\path[|-o] (D) edge (B);
\path[-] (D) edge (C);
\end{tikzpicture}
\\
case 3.1&case 3.2\\
\\
\begin{tikzpicture}[inner sep=1mm]
\node at (-1,0) (A) {$A$};
\node at (1,0) (B) {$B$};
\node at (2,0) (C) {$C$};
\node at (1,-1) (D) {$D$};
\node at (0,-1) (E) {$\ldots$};
\path[|-o] (A) edge (B);
\path[-] (B) edge (C);
\path[-] (A) edge [bend left] (C);
\path[|-o] (A) edge [bend right] (E);
\path[|-o] (E) edge (D);
\path[|-o] (D) edge (B);
\path[o-|] (D) edge (C);
\end{tikzpicture}
&
\begin{tikzpicture}[inner sep=1mm]
\node at (-1,0) (A) {$A$};
\node at (1,0) (B) {$B$};
\node at (2,0) (C) {$C$};
\node at (1,-1) (D) {$D$};
\node at (0,-1) (E) {$\ldots$};
\path[|-o] (A) edge (B);
\path[-] (B) edge (C);
\path[-] (A) edge [bend left] (C);
\path[|-o] (A) edge [bend right] (E);
\path[|-o] (E) edge (D);
\path[|-o] (D) edge (B);
\path[|-] (D) edge (C);
\end{tikzpicture}\\
case 3.3&case 3.4
\end{tabular}}
\end{table}

Note that $C$ cannot belong to the route $A \bo \cdots $$\bo D$ because, otherwise, R3 could not have been applied since the cycle \mbox{$A \bo \cdots \bo D \bo B \no A$} would not have been chordless.

\begin{description}

\item[Case 3.1] If $B \notin S_{CD}$ then $B \nb C$ is in $H$ by R1 in line 4 of Table \ref{tab:algorithmR}, else $B \bn C$ is in $H$ by R2. Either case is a contradiction.

\item[Case 3.2] Note that
\begin{tabular}{c}
\scalebox{0.75}{
\begin{tikzpicture}[inner sep=1mm]
\node at (0,0) (A) {$D$};
\node at (1,0) (B) {$B$};
\node at (2,0) (C) {$C$};
\path[|-o] (A) edge (B);
\path[-] (B) edge (C);
\path[-] (A) edge [bend left] (C);
\end{tikzpicture}}
\end{tabular}
cannot be an induced subgraph of $H$ after line 11 because, otherwise, it would contradict the assumption that
\begin{tabular}{c}
\scalebox{0.75}{
\begin{tikzpicture}[inner sep=1mm]
\node at (0,0) (A) {$A$};
\node at (1,0) (B) {$B$};
\node at (2,0) (C) {$C$};
\path[|-o] (A) edge (B);
\path[-] (B) edge (C);
\path[-] (A) edge [bend left] (C);
\end{tikzpicture}}
\end{tabular}
is one of the first induced subgraph of that form that appeared during the execution of lines 10-11. So, this case is impossible.

\item[Case 3.3] Note that $B \nb C$ is in $H$ by R3, which is a contradiction.

\item[Case 3.4] Note that $C$ cannot be adjacent to any node of the route $A \bo \cdots \bo D$ besides $A$ and $D$ and, thus, $A \bn C$ is in $H$ by R3. To see it, assume to the contrary that $C$ is adjacent to some nodes $E_1, \ldots, E_n \neq A, D$ of the route $A \bo \cdots \bo D$. Assume without loss of generality that $E_i$ is closer to $A$ in the route than $E_{i+1}$ for all $1 \leq i < n$. Now, note that $E_n \bo C$ must be in $H$ by R3. This implies that $E_{n-1} \bo C$ must be in $H$ by R3. By repeated application of this argument, we can conclude that $E_1 \bo C$ must be in $H$ and, thus, $A \bn C$ must be in $H$ by R3, which is a contradiction.

\end{description}

\end{description}

\end{proof}

\begin{lemma}\label{lem:oppositeblockR}
After line 11, every chordless cycle $\rho: V_1, \ldots, V_n=V_1$ in $H$ that has an edge $V_i \bn V_{i+1}$ also has an edge $V_j \nb V_{j+1}$.
\end{lemma}

\begin{proof}
The proof is an adaptation of the proof of Lemma 6 by \citet{Penna2014}. Assume for a contradiction that $\rho$ is of the length three such that $V_{1} \bn V_{2}$ occur and neither $V_{2} \nb V_{3}$ nor $V_{1} \bn V_{3}$ occur. Note that $V_{2} \bb V_{3}$ cannot occur either because, otherwise, $V_{1} \bn V_{3}$ or $V_{1} \bb V_{3}$ must occur by R3. Since the former contradicts the assumption, then the latter must occur. However, this implies that $V_{1} \bb V_{2}$ must occur by R3, which contradicts the assumption. Similarly, $V_{1} \bb V_{3}$ cannot occur either. Then, $\rho$ is of one of the following forms:

\begin{table}[H]
\centering
\scalebox{0.75}{
\begin{tabular}{ccc}
\begin{tikzpicture}[inner sep=1mm]
\node at (0,0) (A) {$V_1$};
\node at (1,0) (B) {$V_{2}$};
\node at (2,0) (C) {$V_{3}$};
\path[|-] (A) edge (B);
\path[-] (B) edge (C);
\path[-] (A) edge [bend left] (C);
\end{tikzpicture}
&
\begin{tikzpicture}[inner sep=1mm]
\node at (0,0) (A) {$V_1$};
\node at (1,0) (B) {$V_{2}$};
\node at (2,0) (C) {$V_{3}$};
\path[|-] (A) edge (B);
\path[o-] (B) edge (C);
\path[-|] (A) edge [bend left] (C);
\end{tikzpicture}
&
\begin{tikzpicture}[inner sep=1mm]
\node at (0,0) (A) {$V_1$};
\node at (1,0) (B) {$V_{2}$};
\node at (2,0) (C) {$V_{3}$};
\path[|-] (A) edge (B);
\path[|-] (B) edge (C);
\path[-] (A) edge [bend left] (C);
\end{tikzpicture}
\end{tabular}}
\end{table}

The first form is impossible by Lemma \ref{lem:noundirectedcycleS}. The second form is impossible because, otherwise, $V_{2} \ob V_{3}$ would occur by R3. The third form is impossible because, otherwise, $V_{1} \bn V_{3}$ would be occur by R3. Thus, the lemma holds for cycles of length three.

Assume for a contradiction that $\rho$ is of length greater than three and has an edge $V_{i} \bn V_{i+1}$ but no edge $V_{j} \nb V_{j+1}$. Note that if \mbox{$V_l \bo V_{l+1} \oo V_{l+2}$} is a subroute of $\rho$, then either $V_{l+1} \bo V_{l+2}$ or $V_{l+1} \nb V_{l+2}$ is in $\rho$ by R1 and R2. Since $\rho$ has no edge $V_j \nb V_{j+1}$, $V_{l+1} \bo V_{l+2}$ is in $\rho$. By repeated application of this reasoning together with the fact that $\rho$ has an edge $V_i \bn V_{i+1}$, we can conclude that every edge in $\rho$ is $V_k \bo V_{k+1}$. Then, by repeated application of R3, observe that every edge in $\rho$ is $V_{k} \bb V_{k+1}$, which contradicts the assumption.
\end{proof}

\begin{lemma}\label{lem:nosd}
After line 11, $H$ can be oriented as indicated in line 14 without creating a semidirected cycle.
\end{lemma}

\begin{proof}
Assume to the contrary that the orientation produces a semidirected cycle $\rho:V_1, \ldots, V_n$. Note that $\rho$ must have a chord because, otherwise, $\rho$ is impossible by Lemma \ref{lem:oppositeblockR}. Specifically, let the chord be between $V_i$ and $V_j$ with $i < j$. Then, divide $\rho$ into the cycles $\rho_L: V_1, \ldots, V_i, V_j, \ldots, V_n=V_1$ and $\rho_R: V_i, \ldots, V_j, V_i$. Note that $\rho_L$ or $\rho_R$ is a semidirected cycle but shorter than $\rho$. By repeated application of this reasoning, we can conclude that the orientation produces a chordless semidirected cycle, which contradicts Lemma \ref{lem:oppositeblockR}.
\end{proof}

\begin{lemma}\label{lem:triplexS}
After line 11, $H$ can be oriented as indicated in line 14 without creating a triplex that is not in $G^*$.
\end{lemma}

\begin{proof}
We call pretriplex to an induced subgraph of $G^*$ or $H$ that results in a triplex when $G^*$ or $H$ are oriented as indicated in line 14. Note that $G^*$ and $H$ have the same pretriplexes after line 9. Assume to the contrary that after line 11 $H$ has a pretriplex that is not in $G^*$. Assume that the spurious pretriplex is created in line 10 when $A \bn B$ becomes $A \bb B$. Then, after line 11 $H$ has a pretriplex (1) $A \bb B \nb C$ or (2) $C \bn A \bb B$. Case (1) implies that $H$ has actually an induced subgraph $A \bb B \bb C$ by R2, which is a contradiction. To see that R2 is applicable, note that $B \in S_{AC}$ because $G^*$ does not have a triplex $(A,B,C)$. Case (2) implies that $H$ has actually an induced subgraph $C \bb A \bb B$ by R2, which again is a contradiction. As before, R2 is clearly applicable. Finally, assume that the spurious pretriplex is created in line 11. Then, after line 11 $H$ has an induced subgraph (1) $A \bn B \nb C$, (2) $A \bn B - C$ or (3) $A \bn B \bb C$. However, this implies that $H$ has actually an induced subgraph $A \bb B \bb C$ or $A \bn B \bn C$ by R2, which again is a contradiction. As before, R2 is clearly applicable.
\end{proof}

\begin{lemma}\label{lem:nsu}
After line 14, the undirected edges in $G^*$ that had no blocks after line 7 are not strong.
\end{lemma}

\begin{proof}
The proof is an adaptation of the proof of Theorem 11 by \citet{SonntagandPenna2015}. Let $F$ denote the graph that contains all and only the edges of $G^*$ resulting from the replacements in line 14, and let $U$ denote the graph that contains the rest of the edges of $G^*$ after line 14. Note that all the edges in $U$ are undirected and they had no blocks when line 14 was to be executed. Therefore, $U$ has no cycle of length greater than three that is chordless by line 5. In other words, $U$ is chordal. Then, we can orient all the edges in $U$ without creating triplexes nor directed cycles by using, for instance, the maximum cardinality search (MCS) algorithm \cite[p. 312]{KollerandFriedman2009}. Consider any such orientation of the edges in $U$ and denote it $D$. Now, add all the edges in $D$ to $F$. As we show below, this last step does not create any triplex or semidirected cycle in $F$:

\begin{itemize}
\item It does not create a triplex $(A,B,C)$ in $F$ because, otherwise, $A - B \ob C$ must exist in $G^*$ when line 14 was to be executed, which implies that $A \bo B$ or $A \ob B$ was in $G^*$ by R1 or R2 when line 14 was to be executed, which contradicts that $A - B$ is in $U$.

\item Assume to the contrary that it does create a semidirected cycle $\rho$ in $F$. We can assume without loss of generality that $\rho$ is chordless because if it has a chord between $V_i$ and $V_j$ with $i < j$. Then, divide $\rho$ into the cycles $\rho_L: V_1, \ldots, V_i, V_j, \ldots, V_n=V_1$ and $\rho_R: V_i, \ldots, V_j, V_i$. Note that $\rho_L$ or $\rho_R$ is a semidirected cycle but shorter than $\rho$. By repeated application of this reasoning, we can conclude that $F$ has a chordless semidirected cycle.

Since $D$ has no directed cycles, $\rho$ must have a $\bn$ or $\bb$ edge when line 14 was to be executed. The former case is impossible \cite[Lemma 10]{SonntagandPenna2015}. The latter case implies that $A - B \bb C$ must exist in $G^*$ when line 14 was to be executed, which implies that $A$ and $C$ are adjacent in $G^*$ because, otherwise, $A \bo B$ or $A \ob B$ was in $G^*$ by R1 or R2 when line 14 was to be executed, which contradicts that $A - B$ is in $U$. Then, $A \bo C$ or $A \ob C$ exists in $G^*$ when line 14 was to be executed \cite[Lemma 9]{SonntagandPenna2015}, which implies that $A \bo B$ or $A \ob B$ was in $G^*$ by R3 when line 14 was to be executed, which contradicts that $A - B$ is in $U$.
\end{itemize}

Consequently, $F$ is a CG that is Markov equivalent to $G$. Finally, let us recall how the MCS algorithm works. It first unmarks all the nodes in $U$ and, then, iterates through the following step until all the nodes are marked: Select any of the unmarked nodes with the largest number of marked neighbors and mark it. Finally, the algorithm orients every edge in $U$ away from the node that was marked earlier. Clearly, any node may get marked first by the algorithm because there is a tie among all the nodes in the first iteration, which implies that every edge may get oriented in any of the two directions in $D$ and thus in $F$. Therefore, either orientation of every edge of $U$ occurs in some CG $F$ that is Markov equivalent to $G$. Then, every edge of $U$ must be a strong undirected edge in $G^*$.
\end{proof}

\begin{theorem}\label{the:correct}
Table \ref{tab:algorithmS} identifies all and only the strong edges in $G^*$.
\end{theorem}

\begin{proof}
By definition of EG, the edges in $G^*$ with blocks on both ends in line 7 correspond to strong undirected edges in $G^*$ after line 14. Moreover, the edges in $G^*$ with no blocks in line 7 correspond to non-strong undirected edges in $G^*$ after line 14, by Lemma \ref{lem:nsu}.

After line 11, $H$ can be oriented as indicated in line 14 without creating semidirected cycles by Lemma \ref{lem:nosd}, and without creating a triplex that is not in $G^*$ by Lemma \ref{lem:triplexS}. Therefore, if $H$ can be oriented as indicated in line 14 without destroying any of the triplexes in $G^*$, then the algorithm has found a CG that is Markov equivalent to $G^*$ and such that $X \ra Y$ is in $G^*$ but $X - Y$ is in the CG found and, thus, $X \ra Y$ is non-strong in $G^*$. Otherwise, $X \ra Y$ is strong in $G^*$. This is checked in line 12.
\end{proof}

The algorithm in Table \ref{tab:algorithmS} may be sped up with the help of the rules in Table \ref{tab:rulesS}. S1-3 should be run while possible before line 8, and S4-6 should be run while possible after line 8 to propagate the labellings due to line 13 in the previous iteration.

\begin{table}[t]
\caption{Rules for accelerating the search for strong directed edges in an EG. The antecedents represent induced subgraphs.}\label{tab:rulesS}
\begin{center}
\scalebox{0.75}{
\begin{tabular}{|cccc|}
\hline
S1:&
\begin{tabular}{c}
\begin{tikzpicture}[inner sep=1mm]
\node at (0,0.5) (A) {$A$};
\node at (0,-0.5) (B) {$B$};
\node at (1,0) (C) {$C$};
\node at (2,0) (D) {$D$};
\path[|-] (A) edge (C);
\path[|-] (B) edge (C);
\path[|-] (C) edge (D);
\end{tikzpicture} 
\end{tabular}
& $\Rightarrow$ &
$C \bn D$ is strong\\
&&&\\
\hline
&&&\\
S2:&
\begin{tabular}{c}
\begin{tikzpicture}[inner sep=1mm]
\node at (0,0) (A) {$A$};
\node at (1,0) (B) {$B$};
\node at (2,0) (C) {$C$};
\path[|-] (A) edge (B);
\path[|-|] (B) edge (C);
\end{tikzpicture} 
\end{tabular}
& $\Rightarrow$ &
$A \bn B$ is strong\\
&&&\\
\hline
&&&\\
S3:&
\begin{tabular}{c}
\begin{tikzpicture}[inner sep=1mm]
\node at (0,0) (A) {$A$};
\node at (2,0) (B) {$B$};
\node at (0,1) (C) {$C$};
\node at (2,1) (D) {$D$};
\node at (1,1) (Q) {$\ldots$};
\path[|-] (A) edge (B);
\path[|-o] (A) edge (C);
\path[|-o] (C) edge (Q);
\path[|-o] (Q) edge (D);
\path[|-] (D) edge (B);
\end{tikzpicture} 
\end{tabular}
& $\Rightarrow$ &
$A \bn B$ is strong\\
&&&\\
\hline
&&&\\
S4:&
\begin{tabular}{c}
$A \bn B \bn C$\\
and $A \bn B$ is strong\\
\end{tabular}
& $\Rightarrow$ &
$B \bn C$ is strong\\
&&&\\
\hline
&&&\\
S5:&
\begin{tabular}{c}
\begin{tikzpicture}[inner sep=1mm]
\node at (0,0) (A) {$A$};
\node at (1,0) (B) {$B$};
\node at (0.5,1) (C) {$C$};
\path[|-] (A) edge (B);
\path[|-o] (A) edge (C);
\path[|-] (C) edge (B);
\end{tikzpicture}\\
and $C \bn B$ is strong\\
\end{tabular}
& $\Rightarrow$ &
$A \bn B$ is strong\\
&&&\\
\hline
&&&\\
S6:&
\begin{tabular}{c}
\begin{tikzpicture}[inner sep=1mm]
\node at (0,0) (A) {$A$};
\node at (1,0) (B) {$B$};
\node at (0.5,1) (C) {$C$};
\path[|-] (A) edge (B);
\path[|-] (A) edge (C);
\path[|-o] (C) edge (B);
\end{tikzpicture}\\
and $A \bn C$ is strong\\
\end{tabular}
& $\Rightarrow$ &
$A \bn B$ is strong\\
\hline
\end{tabular}}
\end{center}
\end{table}

\begin{corollary}
Applying the rules in Table \ref{tab:rulesS} to an EG $G^*$ correctly identifies strong directed edges in $G^*$.
\end{corollary}

\begin{proof}
Consider any member $G$ of the equivalence class of $G^*$. Consider the rule S1. Since $G^*$ has a triplex $(A,C,B)$ after line 14, $G$ must have an edge $A \ra C$ or $B \ra C$. In either case $G$ must also have an edge $C \ra D$, since $G^*$ has not a triplex $(A,C,D)$ or $(B,C,D)$.

Consider the rule S2. Since $G^*$ has a triplex $(A,B,C)$ after line 14 and $G$ has an edge $B - C$ due to the blocks at $B$ and $C$, then $G$ must also have an edge $A \ra B$.

Consider the rule S3. Assume to the contrary that $G$ has an edge $A - B$. Then, $G$ must have an edge $D \ra B$ since $G^*$ has a triplex $(A,B,D)$ after line 14. However, this implies that $G$ has a semidirected cycle due to the blocks in the antecedent of the rule, which is a contradiction.

Consider the rule S4. Since $G^*$ has not a triplex $(A,B,C)$ after line 14 and $G$ has an edge $A \ra B$ because it is strong, then $G$ must also have an edge $B \ra C$.

Consider the rule S5. Since $G$ has an edge $C \ra B$ because it is strong, then $G$ must also have an edge $A \ra B$ to avoid having a semidirected cycle, because either $A \ra C$ or $A - C$ is in $G$ due to the blocks in the antecedent of the rule. The rule S6 can be proven similarly.
\end{proof}

The rules in Table \ref{tab:rulesS} are by no means complete, i.e. there may be strong edges that the rules alone do not detect. Thus, additional rules can be created. We doubt though that a complete set of concise rules can be produced. The difficulty lies in the disjunctive nature of some labellings. For instance, let an EG $G^*$ have induced subgraphs $A \ra C \la B$, $A \ra C \ra \cdots \ra D \ra E$ and $B \ra C \ra \cdots \ra D \ra E$. Since $G^*$ has no triplex in $A \ra C \ra \cdots \ra D \ra E$, if a member $G$ of the equivalence class of $G^*$ has an edge $A \ra C$ then it has an edge $D \ra E$. Similarly, if $G$ has an edge $B \ra C$ then it has an edge $D \ra E$. Then, $G$ has an edge $D \ra E$ because it has an edge $A \ra C$ or $B \ra C$, since $G^*$ and thus $G$ has a triplex $(A,C,B)$. Therefore, $D \ra E$ is strong. Although it is easy to produce a rule for this example, many more such disjunctive examples exist and we do not see any way to produce concise rules for all of them.

\section{CAUSAL EFFECT BOUNDS}\label{sec:bounds}

When the true CG is unknown, a causal effect of the form $p(y | do(x))$ with $X, Y \in V$ cannot be computed, but it can be bounded as follows:
\begin{enumerate}
\item Obtain all the CGs that are Markov equivalent to the true one by running the learning algorithm developed by \citet{Penna2014,Penna2016} or \citet{PennaandGomezOlmedo2016}.

\item Compute the causal effect for each CG obtained as follows. Like in a Bayesian network, any causal effect in a CG $G$ is computable uniquely from observed quantities (i.e. it is identifiable) by adjusting for the appropriate variables. Specifically,
\[
p(y | do(x)) = \int p(y | x, z) p(z) dz
\]
where $Z = Ne_{G}(X) \cup Pa_{G}(X \cup Ne_{G}(X))$ and $Y \notin Z$. The role of $Z$ is to block every non-causal path in $G$ between $X$ and $Y$. We call $Z$ the adjusting set in $G$.
\end{enumerate}
Unfortunately, the learning algorithm in step 1 may be too time consuming for all but small domains. At least, this is the conclusion that follows from the experimental results reported by \citet{Sonntagetal.2015} for a similar algorithm for learning Lauritzen-Wermuth-Frydenberg CGs. Instead, we propose the following alternative approach:
\begin{enumerate}
\item[1'.] Learn the EG $G^*$ corresponding to the true CG from data as follows. First, learn a CG from data as shown by \citet{Penna2014,Penna2016} and \citet{PennaandGomezOlmedo2016} and, then, transform it into an EG as shown by \citet[Section 3]{SonntagandPenna2015}.

\item[2'.] Enumerate all the CGs that are Markov equivalent to $G^*$ as shown by \citet[Theorem 3]{SonntagandPenna2015}.

\item[3'.] Compute the causal effect for each CG enumerated as shown above.
\end{enumerate}
This approach has successfully been applied when the causal models are represented by other graphical models than CGs \citep{Hyttinenetal.2015,MalinskyandSpirtes2016,Maathuisetal.2009}. The experimental results reported by \citet{PennaandGomezOlmedo2016} indicate that the learning algorithm in step 1' scales to medium sized domains. However, the enumeration in step 2' may be too time consuming for all but small domains. Alternatively, we may try to enumerate the adjusting sets in the equivalent CGs without enumerating these explicitly. Specifically, we know that all the adjusting sets are subsets of $Ad_{G^*}(X) \cup Ad_{G^*}(Ad_{G^*}(X))$, because all the equivalent CGs have the same adjacencies as $G^*$. Therefore, we can adjust for every subset of $Ad_{G^*}(X) \cup Ad_{G^*}(Ad_{G^*}(X))$ to obtain bounds for the causal effect of interest. True that some of these subsets are not valid adjusting sets in the sense that they do not correspond to any of the equivalent CGs. However, this does not make the bounds invalid, just more loose. The rest of the section studies a case where all and only the valid adjusting sets can be enumerated efficiently.

Assume that we believe a priori that the dependencies in the domain at hand are due to causal rather than non-causal relationships. Then, we believe a posteriori that the true CG is a maximally oriented CG, because such CGs have the fewest undirected edges in the equivalence class of the EG $G^*$ learned from the data in step 1'. Moreover, recall from Section \ref{sec:minimallyoriented} that all of them have the same undirected edges. Therefore, we can bound the causal effect $p(y | do(x))$ by modifying the latter framework above so that only maximally oriented CGs are enumerated in step 2'. A maximally oriented CG that is equivalent to $G^*$ can be obtained from $G^*$ by repeatedly performing feasible splits \citep[Theorem 3]{SonntagandPenna2015}. Unfortunately, this enumeration method may be inefficient for all but small domains. Instead, we show below how to enumerate the adjusting sets in the maximally oriented CGs that are equivalent to $G^*$ without enumerating these explicitly.

Given a node $X \in V$, we define $St_{G^*}(X) = \{A | A - X$ is a strong edge in $G^* \}$ and $Nst_{G^*}(X) = \{A | A - X$ is a non-strong edge in $G^* \}$. Given a set $S \subseteq Nst_{G^*}(X)$, we let $G^*_{S \ra X}$ denote the graph that is obtained from $G^*$ by replacing the edge $A - X$ with $A \ra X$ for all $A \in S$, and replacing the edge $A - X$ with $A \la X$ for all $A \in Nst_{G^*}(X) \setminus S$. Moreover, we say that $G^*_{S \ra X}$ is locally valid if $G^*_{S \ra X}$ does not have any triplex $(A,X,B)$ that is not in $G^*$. The next theorem proves that producing the adjusting sets in the equivalent maximally oriented CGs simplifies to produce locally valid sets.

\begin{theorem}
$G^*_{S \ra X}$ is locally valid if and only if there is a maximally oriented CG $G$ that is equivalent to $G^*$ and such that $Ne_G(X)=St_{G^*}(X)$ and $Pa_G(X)=Pa_{G^*}(X) \cup S$, which implies that the adjusting set in $G$ is $St_{G^*}(X) \cup Pa_{G^*}(X \cup St_{G^*}(X)) \cup S$.
\end{theorem}

\begin{proof}
The proof is an adaptation of the proof of Lemma 3.1 by \citet{Maathuisetal.2009}. The if part is trivial. To prove the only if part, note first that $S \cup X$ is a complete set because, otherwise, $G^*_{S \ra X}$ would not be locally valid.

Let $G$ denote the graph that contains all and only the non-strong undirected edges in $G^*$. Recall from Lemma \ref{lem:nsu} that these edges had no blocks when line 14 in Table \ref{tab:algorithmS} was to be executed. Therefore, $G$ is chordal by line 5 in Table \ref{tab:algorithmR}. We now show that we can orient the edges of $G$ without creating triplexes or directed cycles and such that $Pa_G(X)=S$. Specifically, we show that there is a perfect elimination sequence that ends with $X$ followed by the nodes in $S$. Orienting the edges of $G$ according to this sequence produces the desired graph. If $G$ is complete, then the sequence clearly exists. If $G$ is not complete, then note that $G$ has at least two non-adjacent simplicial nodes \citep[Theorem 4.1]{JensenandNielsen2007}. Note that one of them is outside of $S \cup X$ because, as shown above, the latter is a complete set. Take that node as the first node in the sequence. Note moreover that the subgraph of $G$ induced by the rest of the nodes is chordal. Therefore, we can repeat the previous step to select the next node in the sequence until we obtain the desired perfect elimination sequence.

Finally, consider the oriented $G$ obtained in the previous paragraph, and add to it all the directed edges and strong undirected edges in $G^*$. We now prove that $G$ is the desired CG in the theorem. First, note that $G$ is maximally oriented because all the undirected edges in it are strong in $G^*$. Second, note that if $G^*$ has a triplex $(A,B,C)$ then $A \bn B \ob C$ must be in $G^*$ when line 14 was to be executed, which implies that neither of the edges in the triplex is non-strong undirected in $G^*$, which implies that $G$ has a triplex $(A,B,C)$. Third, note that $G$ does not have a triplex $(A,B,C)$ that is not in $G^*$ because, otherwise, the triplex should have been created as a product of the perfect elimination sequence above. This is possible only if $A - B - C$ or $A - B \ob C$ exists in $G^*$ when line 14 was to be executed. The former case is impossible by definition of perfect elimination sequence. The latter case implies that $A \bo B$ or $A \ob B$ was in $G^*$ by R1 or R2 when line 14 was to be executed, which contradicts that $A - B$ was a non-strong undirected edge in $G^*$. Fourth, assume to the contrary that $G$ has a semidirected cycle $\rho:V_1, \ldots, V_n$. We can assume without loss of generality that $\rho$ is chordless because if it has a chord between $V_i$ and $V_j$ with $i < j$. Then, divide $\rho$ into the cycles $\rho_L: V_1, \ldots, V_i, V_j, \ldots, V_n=V_1$ and $\rho_R: V_i, \ldots, V_j, V_i$. Note that $\rho_L$ or $\rho_R$ is a semidirected cycle but shorter than $\rho$. By repeated application of this reasoning, we can conclude that $G$ has a chordless semidirected cycle. Note that it follows from the paragraph above that $\rho$ cannot consists of just non-strong undirected edges in $G^*$. Then, it includes some edge that was $A \bn B$ or $A \bb B$ when line 14 was to be executed. The former alternative is impossible \cite[Lemma 10]{SonntagandPenna2015}. The latter alternative implies that $A \bb B - C$ must exist in $G^*$ when line 14 was to be executed, which implies that $A$ and $C$ are adjacent in $G^*$ because, otherwise, $B \bo C$ or $B \ob C$ was in $G^*$ by R1 or R2 when line 14 was to be executed, which contradicts that $B - C$ is a non-strong undirected edge in $G^*$. Then, $A \bo C$ or $A \ob C$ exists in $G^*$ when line 14 was to be executed \cite[Lemma 9]{SonntagandPenna2015}, which implies that $B \bo C$ or $B \ob C$ was in $G^*$ by R3 when line 14 was to be executed, which contradicts that $B - C$ is a non-strong undirected edge in $G^*$.
\end{proof}

The procedure outlined above can be simplified as follows.

\begin{corollary}
$St_{G^*}(X) = \emptyset$ or $Nst_{G^*}(X) = \emptyset$.
\end{corollary}

\begin{proof}
Assume the contrary. Then, $G^*$ has a subgraph $A \bb X - B$ when line 14 in Table \ref{tab:algorithmS} is to be executed. Then, $A$ and $B$ are adjacent in $G^*$ because, otherwise, the edge $X - B$ would have some block by R1 or R2. However, this implies that the edge $A - B$ has some block by Lemma \ref{lem:noundirectedcycleS}, which implies that $X - B$ has some block by R3. This is a contradiction.
\end{proof}

\section{DISCUSSION}\label{sec:discussion}

In this paper, we have presented an algorithm to identify the strong edges in an EG. We have also shown how this makes it possible to compute bounds of causal effects under the assumption that the true CG is unknown but maximally oriented. In the future, we would like to derive a similar result for minimally oriented CGs. Moreover, as mentioned in the introduction, an EG is a deflagged graph but not necessarily the largest in the equivalence class. Therefore, an EG may contain a directed edge where the largest deflagged graph has an undirected edge. Then, the algorithm in Table \ref{tab:algorithmS} may be improved by consulting the largest deflagged graph before trying labeling a directed edge as strong. An algorithm for constructing this graph exists \citep{RoveratoandStudeny2006}.


\begin{thebibliography}{}

\bibitem[Andersson et al., 2001]{Anderssonetal.2001}
Andersson, S. A., Madigan, D. and Perlman, M. D. Alternative Markov Properties for Chain Graphs. {\em Scandinavian Journal of Statistics}, 28:33-85, 2001.

\bibitem[Andersson and Perlman, 2004]{AnderssonandPerlman2004}
Andersson, S. A. and Perlman, M. D. Characterizing Markov Equivalent Classes for AMP Chain Graph Models. Technical Report 453, University of Washington, 2004. Available at http://www.stat.washington.edu/www/research/reports
/2004/tr453.pdf.

\bibitem[Andersson and Perlman, 2006]{AnderssonandPerlman2006}
Andersson, S. A. and Perlman, M. D. Characterizing Markov Equivalent Classes for AMP Chain Graph Models. {\em The Annals of Statistics}, 34:939-972, 2006.

\bibitem[Koller and Friedman, 2009]{KollerandFriedman2009}
Koller, D. and Friedman, N. {\em Probabilistic Graphical Models}. MIT Press, 2009.

\bibitem[Hyttinen et al., 2015]{Hyttinenetal.2015}
Hyttinen, A., Eberhardt, F. and J\"arvisalo, M. Do-calculus when the True Graph is Unknown. In {\em Proceedings of the 31th Conference on Uncertainty in Artificial Intelligence}, 395-404, 2015.

\bibitem[Jensen and Nielsen, 2007]{JensenandNielsen2007}
Jensen, F. V. and Nielsen, T. D. {\em Bayesian Networks and Decision Graphs}. Springer Verlag, 2007.

\bibitem[Levitz et al., 2001]{Levitzetal.2001}
Levitz, M., Perlman M. D. and Madigan, D. Separation and Completeness Properties for AMP Chain Graph Markov Models. {\em The Annals of Statistics}, 29:1751-1784, 2001.

\bibitem[Maathuis et al., 2009]{Maathuisetal.2009}
Maathuis, M. H., Kalisch, M. and B\"uhlmann, P. Estimating High-Dimensional Intervention Effects from Observational Data. {\em The Annals of Statistics}, 37:3133-3164, 2009.

\bibitem[Malinsky and Spirtes, 2016]{MalinskyandSpirtes2016}
Malinsky, D. and Spirtes, P. Estimating Causal Effects with Ancestral Graph Markov Models. In {\em Proceedings of the 8th International Conference on Probabilistic Graphical Models}, 299-309, 2016.

\bibitem[Pearl, 2009]{Pearl2009}
Pearl, J. {\em Causality: Models, Reasoning, and Inference}. Cambridge University Press, 2009.

\bibitem[Pe\~{n}a, 2014]{Penna2014}
Pe\~{n}a, J. M. Learning AMP Chain Graphs and some Marginal Models Thereof under Faithfulness. {\em International Journal of Approximate Reasoning}, 55:1011-1021, 2014.

\bibitem[Pe\~{n}a, 2016]{Penna2016}
Pe\~{n}a, J. M. Alternative Markov and Causal Properties for Acyclic Directed Mixed Graphs. In {\em Proceedings of the 32nd Conference on Uncertainty in Artificial Intelligence}, 577-586, 2016.

\bibitem[Pe\~{n}a and G\'omez-Olmedo, 2016]{PennaandGomezOlmedo2016}
Pe\~{n}a, J. M. and G\'omez-Olmedo, M. Learning Marginal AMP Chain Graphs under Faithfulness Revisited. {\em International Journal of Approximate Reasoning}, 68:108-126, 2016.

\bibitem[Peters et al., 2017]{Petersetal.2017}
Peters, J., Janzing, D. and Sch\"{o}lkopf, B. {\em Elements of Causal Inference: Foundations and Learning Algorithms}. The MIT Press, 2017.

\bibitem[Roverato and Studen\'{y}, 2006]{RoveratoandStudeny2006}
Roverato, A. and Studen\'{y}, M. A Graphical Representation of Equivalence Classes of AMP Chain Graphs. {\em Journal of Machine Learning Research}, 7:1045-1078, 2006.

\bibitem[Sonntag and Pe\~{n}a, 2015]{SonntagandPenna2015}
Sonntag, D. and Pe\~{n}ña, J. M. Chain Graph Interpretations and their Relations Revisited.  {\em International Journal of Approximate Reasoning}, 58:39-56, 2015.

\bibitem[Sonntag et al., 2015]{Sonntagetal.2015}
Sonntag, D., J\"arvisalo, M., Pe\~{n}ña, J. M. and Hyttinen, A. Learning Optimal Chain Graphs with Answer Set Programming. In {\em Proceedings of the 31st Conference on Uncertainty in Artificial Intelligence}, 822-831, 2015.

\end{thebibliography}
\end{document}